\documentclass[12pt]{amsart}
\usepackage[T1]{fontenc}
\usepackage{graphics}
\usepackage{amsfonts,amssymb,color}
\usepackage[mathscr]{eucal}
\usepackage{amsmath, amsthm}
\usepackage{mathrsfs}
\usepackage{amsbsy}
\usepackage{dsfont}
\usepackage{bbm}
\usepackage{wasysym}
\usepackage{stmaryrd}
\usepackage{url}

\input xypic
\xyoption{all}

\makeindex
\makeglossary

\begin{document}
\baselineskip = 16pt

\newcommand \ZZ {{\mathbb Z}}
\newcommand \NN {{\mathbb N}}
\newcommand \RR {{\mathbb R}}
\newcommand \PR {{\mathbb P}}
\newcommand \AF {{\mathbb A}}
\newcommand \GG {{\mathbb G}}
\newcommand \QQ {{\mathbb Q}}
\newcommand \CC {{\mathbb C}}
\newcommand \bcA {{\mathscr A}}
\newcommand \bcC {{\mathscr C}}
\newcommand \bcD {{\mathscr D}}
\newcommand \bcF {{\mathscr F}}
\newcommand \bcG {{\mathscr G}}
\newcommand \bcH {{\mathscr H}}
\newcommand \bcM {{\mathscr M}}
\newcommand \bcI {{\mathscr I}}
\newcommand \bcJ {{\mathscr J}}
\newcommand \bcK {{\mathscr K}}
\newcommand \bcL {{\mathscr L}}
\newcommand \bcO {{\mathscr O}}
\newcommand \bcP {{\mathscr P}}
\newcommand \bcQ {{\mathscr Q}}
\newcommand \bcR {{\mathscr R}}
\newcommand \bcS {{\mathscr S}}
\newcommand \bcV {{\mathscr V}}
\newcommand \bcU {{\mathscr U}}
\newcommand \bcW {{\mathscr W}}
\newcommand \bcX {{\mathscr X}}
\newcommand \bcY {{\mathscr Y}}
\newcommand \bcZ {{\mathscr Z}}
\newcommand \goa {{\mathfrak a}}
\newcommand \gob {{\mathfrak b}}
\newcommand \goc {{\mathfrak c}}
\newcommand \gom {{\mathfrak m}}
\newcommand \gon {{\mathfrak n}}
\newcommand \gop {{\mathfrak p}}
\newcommand \goq {{\mathfrak q}}
\newcommand \goQ {{\mathfrak Q}}
\newcommand \goP {{\mathfrak P}}
\newcommand \goM {{\mathfrak M}}
\newcommand \goN {{\mathfrak N}}
\newcommand \uno {{\mathbbm 1}}
\newcommand \Le {{\mathbbm L}}
\newcommand \Spec {{\rm {Spec}}}
\newcommand \Gr {{\rm {Gr}}}
\newcommand \Pic {{\rm {Pic}}}
\newcommand \Jac {{{J}}}
\newcommand \Alb {{\rm {Alb}}}
\newcommand \Corr {{Corr}}
\newcommand \Chow {{\mathscr C}}
\newcommand \Sym {{\rm {Sym}}}
\newcommand \Prym {{\rm {Prym}}}
\newcommand \cha {{\rm {char}}}
\newcommand \eff {{\rm {eff}}}
\newcommand \tr {{\rm {tr}}}
\newcommand \Tr {{\rm {Tr}}}
\newcommand \pr {{\rm {pr}}}
\newcommand \ev {{\it {ev}}}
\newcommand \cl {{\rm {cl}}}
\newcommand \interior {{\rm {Int}}}
\newcommand \sep {{\rm {sep}}}
\newcommand \td {{\rm {tdeg}}}
\newcommand \alg {{\rm {alg}}}
\newcommand \im {{\rm im}}
\newcommand \gr {{\rm {gr}}}
\newcommand \op {{\rm op}}
\newcommand \Hom {{\rm Hom}}
\newcommand \Hilb {{\rm Hilb}}
\newcommand \Sch {{\mathscr S\! }{\it ch}}
\newcommand \cHilb {{\mathscr H\! }{\it ilb}}
\newcommand \cHom {{\mathscr H\! }{\it om}}
\newcommand \colim {{{\rm colim}\, }} % colimit
\newcommand \End {{\rm {End}}}
\newcommand \coker {{\rm {coker}}}
\newcommand \id {{\rm {id}}}
\newcommand \van {{\rm {van}}}
\newcommand \spc {{\rm {sp}}}
\newcommand \Ob {{\rm Ob}}
\newcommand \Aut {{\rm Aut}}
\newcommand \cor {{\rm {cor}}}
\newcommand \Cor {{\it {Corr}}}
\newcommand \res {{\rm {res}}}
\newcommand \red {{\rm{red}}}
\newcommand \Gal {{\rm {Gal}}}
\newcommand \PGL {{\rm {PGL}}}
\newcommand \Bl {{\rm {Bl}}}
\newcommand \Sing {{\rm {Sing}}}
\newcommand \spn {{\rm {span}}}
\newcommand \Nm {{\rm {Nm}}}
\newcommand \inv {{\rm {inv}}}
\newcommand \codim {{\rm {codim}}}
\newcommand \Div{{\rm{Div}}}
\newcommand \CH{{\rm{CH}}}
\newcommand \sg {{\Sigma }}
\newcommand \DM {{\sf DM}}
\newcommand \Gm {{{\mathbb G}_{\rm m}}}
\newcommand \tame {\rm {tame }}
\newcommand \znak {{\natural }}
\newcommand \lra {\longrightarrow}
\newcommand \hra {\hookrightarrow}
\newcommand \rra {\rightrightarrows}
\newcommand \ord {{\rm {ord}}}
\newcommand \Rat {{\mathscr Rat}}
\newcommand \rd {{\rm {red}}}
\newcommand \bSpec {{\bf {Spec}}}
\newcommand \Proj {{\rm {Proj}}}
\newcommand \pdiv {{\rm {div}}}
\newcommand \wt {\widetilde }
\newcommand \ac {\acute }
\newcommand \ch {\check }
\newcommand \ol {\overline }
\newcommand \Th {\Theta}
\newcommand \cAb {{\mathscr A\! }{\it b}}

\newenvironment{pf}{\par\noindent{\em Proof}.}{\hfill\framebox(6,6)
\par\medskip}

\newtheorem{theorem}[subsection]{Theorem}
\newtheorem{conjecture}[subsection]{Conjecture}
\newtheorem{proposition}[subsection]{Proposition}
\newtheorem{lemma}[subsection]{Lemma}
\newtheorem{remark}[subsection]{Remark}
\newtheorem{remarks}[subsection]{Remarks}
\newtheorem{definition}[subsection]{Definition}
\newtheorem{corollary}[subsection]{Corollary}
\newtheorem{example}[subsection]{Example}
\newtheorem{examples}[subsection]{examples}

\title[Bloch's conjecture]{Bloch's conjecture on certain surfaces of general type  with $p_g=0$ and with an involution: the Enriques case}
\author{Kalyan Banerjee}

\address{VIT Chennai}

\email{kalyan.banerjee@vit.ac.in}

\begin{abstract}
In this short note we prove that an involution on certain examples of surfaces of general type with $p_g=0$, acts as identity on the Chow group of zero cycles of the relevant surface. In particular we consider examples of such surfaces when the quotient is an Enriques surface and show that the Bloch conjecture holds for such surfaces.
\end{abstract}
\maketitle

\section{Introduction}
In \cite{M} Mumford has proved that if the geometric genus of a  smooth algebraic surface is greater than zero then the Chow group of zero cycles on  the surface is not finite dimensional in the sense that the natural map from the symmetric powers of the surface to the Chow group is not surjective. In other words this can be rephrased in terms of the albanese map from the Chow group of zero cycles on the surface to the albanese variety of the surface. It means that the albanese kernel is non-trivial and huge cannot be parametrized by an algebraic variety. So this side of the story is well understood and the it is interesting to look at the converse. That is whether a surface with geometric genus zero has Chow group of zero cycles supported at one point. This conjecture was originally made by Spencer Bloch. It is known for the surfaces of not of general type with geometric genus equal to zero due to \cite{BKL}. After that the conjecture was verified for some examples of surfaces of general type with geometric genus zero  due to \cite{B},\cite{IM}, \cite{Gul1}, \cite{Gul2}, \cite{V}, \cite{VC}.

This conjecture due to Bloch has been extended to the following for arbitrary smooth projective surfaces. Let $S,T$ be two smooth projective surfaces with $\Gamma$ a correspondence supported on $S\times T$. Suppose that $\Gamma^*$ from $H^{2,0}(T)$ to $H^{2,0}(S)$ is zero, then it is conjectured that $\Gamma_*$ from $\CH_0(S)$ to $\CH_0(T)$ vanishes on the albanese kernel of $\CH_0(S)\to Alb(S)$. This conjecture gives Bloch's conjecture when we take $S=T$ and $\Gamma$ to be the diagonal of a surface with geometric genus zero.

Let $S$ be a smooth projective surface equipped with an involution. Then this conjecture was verified when the correspondence is the graph of the involution on the smooth projective surface $S$ due to \cite{HK}, \cite{Voi},\cite{GT}. The aim of this manuscript is to verify the generalised Bloch's conjecture for some examples of surfaces of general type with geometric genus zero. Namely the Numerical Godeaux surfaces with an involution, some numerical Campedelli surfaces with an involution. These surfaces are surfaces of general type with geometric genus equal to $0$ and the self intersection of the canonical class is $1,2$ (for numercial Godeaux and numerical Campedelli surfaces respectively). Due to \cite{CCL}, \cite{CLP}, these surfaces with involution are classified. The Numerical Godeaux surfaces with an involution are classified into two parts, one is the quotient surface is birational to an Enriques surface and the other is that the quotient surface is rational. For Numerical Campedelli surfaces the situation is a bit more involved. The Numerical Campedelli surfaces whose bicanonical map factors through the quotient map arising from the involution are classified into two parts again, the quotient surface is birational to an Enriques surface and the other one is that the quotient surface is rational. These surfaces  are the examples for which the generalised Bloch's conjecture has been verified in this article. So the main theorem of this manuscript are the following.

\begin{theorem}
For all Numerical Godeaux surfaces with an involution such that the quotient is Enriques, the involution acts as identity on the group of algebraically trivial zero cycles modulo rational equivalence.
\end{theorem}

\begin{theorem}
For all Numerical Campedelli surfaces with an involution, such that the bicanonical map factors through the quotient map and the quotient is Enriques, the involution acts as identity on the group of algebraically trivial zero cycles modulo rational equivalence.
\end{theorem}

Consider a surface $S$ of general type with geometric genus equal to zero with $K^2=3$. Suppose such a surface has an involution $i$, such that the bicanonical map is not composed with the involution and the quotient is not rational, then by Theorem 5 of \cite{Ri}, we have the classification for the branch locus of the quotient map. Consider the case when this branch locus contains $5$ curves of self intersection $-1$ and an elliptic curve (in such a case the quotient is a surface not of general type with $p_g=0$). Then taking a very ample divisor on $S/i$ and arguing as in \ref{theorem2} we have:

\begin{theorem}
\label{theorem3}
The Bloch's conjecture holds for the surface $S$ with $K_S^2=3$ and having an involution such that the bicanonical map is not composed with an involution, with the components of the branch locus mentioned above.
\end{theorem}

As an application we show that the Bloch conjecture holds true on certain fake projective planes.

It is in fact interesting to understand whether the surfaces with irregularity zero and with an involution satisfies the same conjecture. The technique used to prove  these results are in  the same line as in \cite{Voi}, where the conjecture is verified for a K3 surface with a symplectic involution. She invokes the notion of finite dimensionality in the Roitman sense \cite{R1} and prove that the correspondence given by the difference of  the graph of the involution and the diagonal is finite dimensional. We use the same technique. Our calculation to show this is  bit different since we are working with a surface of general type with geometric genus equal to zero instead of a K3 surfaces.

{\small \textbf{Acknowledgements:} The author would like to thank the hospitality of IISER-Mohali and VIT Chennai for hosting this project. The author is indebted to Vladimir Guletskii for many useful conversations relevant to the theme of the paper. The author likes to thank Claire Voisin for her  advice on the theme of the paper. Finally the author thanks Claudio Pedrini for suggesting the problem  about surfaces with $K^2=3$ and $p_g=0$ having an involution and for fake projective planes. }

\section{Finite dimensionality in the sense of Roitman and $\CH_0$}

The beginning of this section is recalled from \cite{Voi}[Section 1] for the convenience of the reader and also to keep track of the arguments in the later part of the section of this paper.

First we recall the notion of finite dimensionality in the sense of Roitman \cite{R1}. Let $X$ be a smooth projective variety over $\CC$ and let $P$ be a subgroup of $\CH_0(X)$, then we will say that $P$ is finite dimensional in the Roitman sense, if there exists a smooth projective variety $W$ and a correspondence $\Gamma$ on $W\times X$, such that $P$ is contained in $\{\Gamma_*(w)|w\in W\}$.

The following proposition is due to Roitman and also due to Voisin. For convenience we recall the proof and the theorem.

Let $M,X$ be smooth projective varieties and $Z$ is a correspondence between $M$ and $X$ of codimension $d$.

\begin{theorem}\cite{Voi}
\label{theorem 1}
Assume that the image of $Z_*$ from $\CH_0(M)$ to $\CH_0(X)$ is finite dimensional in the Roitman sense. Then the map $Z_*$ from ${\CH_0(M)}_{hom}$ to $\CH_0(X)$ factors through the albanese map from ${\CH_0(M)}_{hom}$ to $\Alb(M)$.
\end{theorem}

\begin{proof}
By the definition of finite dimensionality in the sense of Roitman there exists a smooth projective variety $W$ and a correspondence $\Gamma$ on $W\times X$, such that image of $Z_*$ is contained in the set $\Gamma_*(W)$. Let $C\subset M$, be a smooth projective curve obtained by taking multisections of $M$, then by Lefschetz hyperplane section theorem the Jacobian of $C$, maps surjectively onto $\Alb(M)$. The kernel $K(C)$ is an abelian variety. Now we prove the following lemma about $K(C)$

\begin{lemma}\cite{Voi}
The abelian variety $K(C)$ is simple for a general multisection.
\end{lemma}
\begin{proof}
We reduce the problem to the case when $M$ is a surface, by taking hyperplane multisections of $M$ and observing that the Albanese remains the same. Meaning that if $T$ is a multisection of $M$, then $\Alb(M)=\Alb(T)$. Consider the embedding of $T$ into the projective space where $M$ is embedded. Then consider a Lefschetz pencil of hyperplane sections on $T$. In this way we obtain the curves $T_t$, for $t\in \PR^1$. Then the Picard-Lefschetz formula says that the monodromy representation of $\pi_1(\PR^1\setminus \{t_1,\cdots,t_m\},t)$ acts irreducibly on the vanishing cohomology
$$\ker(H^1(T_t,\QQ)\to H^3(S,\QQ))\;.$$
The action of $\pi_1(\PR^1\setminus \{t_1,\cdots,t_m\},t)$ on this vanishing cohomology is given in precise terms.
$$\gamma.\alpha=\alpha\pm\langle \alpha,\delta_{\gamma}\rangle \delta_{\gamma}$$
where $\gamma$ is an element of the fundamental group, $\delta_{\gamma}$ is the vanishing cycle corresponding to $\gamma$ and $\langle\rangle$ is the intersection bilinear form. Now we prove that any group $G$ of finite index in the fundamental group also acts in the irreducible way.

Let $G$ be of index $m$, in the fundamental group. Suppose that $V $ is a $G$-stable subspace of the vanishing cohomology. That means that for all $\gamma$ in $G$, and $v\in V$ we have that $\gamma.v$ is again in $V$. So let $\gamma$ belong to the fundamental group, then $\gamma^m$ is in $G$, so we have that
$$\gamma^m.v=v\pm m\langle v,\delta_{\gamma}\rangle \delta_{\gamma}$$
is in $V$, which gives that
$$m\langle v,\delta_{\gamma}\rangle \delta_{\gamma}$$
is in $V$, so
$$\langle v,\delta_{\gamma}\rangle \delta_{\gamma}$$
is in $V$, so we have that $\gamma.v$ is in for all $\gamma$. This proves that $V$ is stable under the action of the fundamental group. Therefore $V$ is either trivial or the whole space of vanishing cohomology. So $G$ acts irreducibly.

Now assume that for a general curve $T_t$, which is smooth, the abelian variety $K(T_t)$ is not simple. That would mean that there exists a non-trivial abelian subvariety $A_t$ of $K(T_t)$, which by the equivalence of Hodge structures and abelian varieties gives rise to a sub-Hodge structure of the vanishing cohomology of $T_t$, with $\QQ$ coefficients. Now  we have the point $t$ a close point on $\mathbb A^1$, so we have a homomorphism from $\CC(t)$ to $\CC$. Therefore we can pull-back $A_t$, over $\Spec(\CC(t))$. To get an abelian variety over function field. Considering a finite extension $L$ of $\CC(t)$, inside $\CC$, this abelian variety, the pull-back of $K(T_t)$ is defined over $L$. Put $L=\CC(D)$, where $D$ is a smooth projective curve mapping finitely onto $\PR^1$. Let $U$ be Zariski open in $D$, and spread $A_t,K(T_t)$, over $U$, to get $\bcA,\bcK$, abelian schemes. Throwing some more points from $U$, we have that $\bcA,\bcK\to U$ is smooth and proper, so by Ehresmann's theorem they are fibrations. So we have that $\pi_1(U,t')$ acts on $H^{2d-1}(A_t,\QQ)$ and on the vanishing cohomology and the inclusion of $H^{2d-1}(A_t,\QQ)$ into the vanishing cohomology is a map of $\pi_1(U,t')$, modules, because it is induced by the regular map from $\bcA$ to $\bcK$. Now $\pi_1(U,t')$ is a finite index subgroup of $\pi_1(\PR^1\setminus\{t_1,\cdots,t_m\},t)$ so the action of $\pi_1(U,t')$ on $H^{2d-1}(A_t,\QQ)$ is irreducible, hence we have $A_t$ is either trivial or all of $K(T_t)$. So $K(T_t)$ is simple.
\end{proof}

We now have a such a $C$ inside $M$, let $j:C\to M$ denote the closed embedding of $C$ into $M$. By the above lemma $K(C)$ is simple. So we have a homomorphism $j_*$ from $J(C)$ to $\CH_0(M)$. Note that for $H_i$ sufficiently ample we can make the dimension of $K(C)$ arbitrarily large so that dimension of $K(C)$ is greater than that of $W$.

Let $R$ be a set inside $K(C)\times W$, given by
$$\{(k,w)|Z_*j_*(k)=\Gamma_*(w)\}$$
it is known that $R$ is a countable union of Zariski closed subsets in the product. Since image of $Z_*$ is contained in $\Gamma_*(W)$, we have that, the first projection of $R$ to $K(C)$ is onto. We write $R=\cup_i R_i$. Since $K(C)$ is irreducible, we have that $R_i$ mapping onto $K(C)$, for some $i$. Then dimension of $R_i$ is greater than or equal to $\dim(K(C))$ which is  strictly greater than dimension of $W$. So we have positive dimensional fibers of the projection $R_i\to W$. So the fiber over a general point $w\in W$, generates $K(C)$, because $K(C)$ is simple. On the other hand for any $f$ in the fiber we have that
$$Z_*j_*(f)=\Gamma_*(w)$$
so for any zero cycle on the fiber we have that
$$Z_*j_*(z)=\deg(z)\Gamma_*(w)$$
since $\deg(z)=0$, for any zero cycle on the fiber we have $Z_*j_*$ vanishes identically on $K(C)$.

Now for any degree zero zero cycle $z_m$ on $M$, we write it as $z_m^+-z_m^-$, where $z_m^+=\sum_i P_1\cdots+P_k$ and $z_m^-=\sum_i P_{k+1}\cdots+P_{2k}$. Sp we have a tuple of points $(P_1,\cdots,P_{2k})$ in $M^{2k}$. We blow up $M$ at these points. Let $\tau:M'\to M$ be the blow up, with exceptional divisors $E_i$ over the point $m_i$. Let $H$ belong to $\Pic(M)$, such that $\gamma^*(H)-n\sum E_i$ is ample on $M'$ and then consider a multiple of $L$ which is very ample. Then we argue as before saying that there exists a curve $C$ in the linear system of $mL$, such that the kernel of the map
$$J(C)\to \Alb(M)$$
is a simple abelian variety of sufficiently large dimension. Now $\tau(C)$ contains all the points $m_i$. Assume that the zero cycle $z_m$ is annihilated by $alb_M$, the any lift of $z_m$ to $M'$ belongs to the kernel of $alb_{M'}$, since $\Alb(M')\to \Alb(M)$ is isomorphism. Let $z'$ be lift of $z_m$ which is supported on $C$ so $z'$ belongs to $j_*(K(C))$. Now we apply the previous argument to $Z'=Z\circ \tau$, (noting that $Z'_*$ has finite dimensional image as its image is contained in that of $Z_*$) between $M'$ and $X$, we get that
$Z_*(z_m)=Z'_*(z')=0$, so the homomorphism $Z_*$ factors through the albanese map.
\end{proof}
\begin{theorem}
\label{theorem2}
Let $S$ be a surface of general type with $p_g=0=q$, let $i:S\to S$ is an involution and $K_S^2=1$ or $2$. Let $S/i$ be birational to an Enriques surface such that the branch locus consists of a unique component of genus two or three and finitely many isolated points. Then the anti-invariant part
$$\CH_0(S)^-=\{z\in \CH_0(S):i_*(z)=-z\}$$
is finite dimensional in the Roitman sense.
\end{theorem}
\begin{proof}
Let $S$ be a surface of general type with $p_g=q=0$ such that the quotient surface $S/i$ is Enriques. Then the fixed point set of the involution is a union of finitely many points and a curve of arithmetic genus $2$ say $D$. Blow up the surface $S$ at the finitely many fixed points of the involution. Denote the surface obtained by blow-up as $\wt{S}$, let $i$ be the involution acting on $\wt{S}$. Consider the quotient $\wt{S}/i$, this is the minimal resolution of singularities of the surface $S/i$. Call this surface as $S'$. Let $L$ be a very ample line bundle on $S'=\wt{S}/i$ and we know that  $2K_S'=0$. By the adjunction formula we have that
$$L^2+L.K=2g-2$$
where $g$ is the genus of a curve in the linear system of $L$. Since $L$ is very ample the left hand side of the above equality is positive, so the genus of the curves in the linear system of $L$ is positive and greater than one. Now we have the exact sequence
$$\bcI_C\to \bcO(S')\to \bcO(S')/\bcI_C\to 0$$
where $C$ is the in the linear system of $L$, $\bcI_C=\bcO(-C)$, now tensor the above sequence with $\bcO(C)$
to get
$$0\to \bcO(S')\to \bcO(C)\to \bcO_C(C)\to 0$$
which gives
$$0\to \CC\to H^0(S',L)\to H^0(C,L|_C)\to 0$$
by the adjunction formula and by using the fact that irregularity of the surface $S'$ is zero. This gives that the dimension of the linear system of $L$ is $g$. Now for $C$ in the linear system $L$ we have the curve $\wt{C}$ in $S$, the inverse image of $C$ is a double cover of $C$ and it is smooth for general $C$. Therefore it is an ramified (along the intersection $D.C$) double cover of $C$. By the Hodge index theorem $\wt{C}$ is connected. Suppose not, then it has two components $C_1,C_2$, so the divisor $C_1+C_2$ is ample, hence
$$(C_1+C_2)^2$$
is greater than zero, which implies that $C_1^2,C_2^2$ greater than zero and $C_1.C_2=0$, since $\wt{C}$ is smooth.  This can only happen if $C_1^2=C_2^2=0$ by Hodge index theorem, which contradicts the
ampleness of the divisor $C_1+C_2$.

Let $\Gamma$ be the correspondence on $S\times S$ given by $\Delta_S-Graph(i)$. We now prove that $\Gamma_*(S^g)=\Gamma_*(S^{g+1})$

Let $s=(s_1,\cdots,s_{g})$ be a general point of $S^{g}$ and let $\sigma_i$ is the image of $s_i$ under the quotient map $S\to S'$. Then a generic $(s_1,\cdots,s_{g})$ gives rise to a generic $(\sigma_1,\cdots,\sigma_{g})$ in $S'^{g}$ and there exists a unique $C_s$ in the linear system $|L|$ such that it contains all the $\sigma_i$. The curve $C_s$ is general in the linear system, so its inverse image under the quotient map is an etale double cover $\wt{C_s} $ of $C_s$, where $\wt{C_s}$ contains all the points $s_i$. The zero cycle
$$z_s=\sum_l s_l -\sum_l i(s_l)$$
is supported on $J(\wt{C_s})$ and it is anti-invariant under $i$ so it is actually in the Prym variety  $P(\wt{C_s}/C_s)$ of the ramified double cover $\wt{C_s}\to C_s$, which is $g-1+\deg(D.C)/2$ dimensional abelian variety and also is the kernel of the norm map from $J(\wt{C_s})$ to $J(C_s)$. So we have the following map
$$(s_1,\cdots,s_g)\mapsto z_s\;,$$
from $S^g\to \CH_0(S)^{-}$
and its factors through a morphism
$$f:U\to \bcP(\wt{\bcC}/\bcC)\times J(D)$$
as follows.
Here $\bcC\to |L|_0$ is the universal smooth curve over the Zariski open set $|L|_0$ which is of dimension $g$ and $\wt{\bcC}\to |L|_0$ is its universal double cover, and $\bcP(\wt{\bcC}/\bcC)$ is the corresponding Prym fibration. So the Prym fibration is of dimension $2g-1+\deg(C.R)/2$. We have to prove that the morphism $f$ has positive dimensional fibers. Consider the subsets
$$S_{D_i}=\{(s_1,\cdots,s_g)|\textit{exactly $i$ of $s_1,\cdots,s_g$ are in $D$}\}$$

Then we can consider the map from $S_{D_i}\cong S^{g-i}\times D^i$ to $\bcP(\wt{\bcC}/\bcC)\times J(D)$. Given a general point $(s_1,\cdots,s_g)$ the image of it is contained in $P(\wt{C_s}/C_s)\times J(D)$. Now the $s_i$'s which are not in $D$, they are in the image of the map $\Sym^{g-i}\wt{C_s}\to J(\wt{C_s})$. So it is contained in a subvariety inside $J(\wt{C_s})$ as well as in $P(\wt{C_s}/C_s)$ which is of dimension $g-1+\deg(D.C_s)/2$. So the dimension of the image is less than both $g-i$ and $g-1+\deg(D.C_s)/2$. In particular the dimension of the image is less than $g-i$.  So the image of $\wt{C_s}^{g-i}\times D^i$ is of dimension atmost $g-i+2$, as it is contained in $P(\wt{C_s}/C_s)\times J(D)$. Therefore the image of  the map from $S^{g-i}\times D^i$ is in the variety given by $\bcP(\wt{\bcC}/\bcC)\times J(D)$ and it is of dimension atmost $g-i+2+g=2g-i+2$ (adding the fiber dimension with the dimension of the base). Now we can stratify the general points of $S^g$ as a point in $S^{g-i}\times D^i$ for some $i$ (ranging from $0$ to $g$). So the dimension of the fiber from $S^{g}\to \bcP(\wt{\bcC}/\bcC)\times J(D)$ (which is the image of the map $S^{g-i}\times D^i$) is atleast
$$\max\{2g-2g+i-2=i-2\}\;.$$
Now given a curve $C_s$, the intersection number $D.C_s\geq 1$. So by taking a large power of the line bundle $L$, we have $i>2$. Therefore the above fiber of the map $S^g\to \bcP(\wt{\bcC}/\bcC)\times J(D)$ is always positive. So it contains a curve. Say $F_s$.

So let $z_s$ be the zero cycles as above, supported on the Prym variety $P(\wt{C_s}/C_s)$. Then there exists a curve $F_s$ in $U$ such that for any $(t_1,\cdots,t_g)$ in $F_s$, the cycle $z_t$ is rationally equivalent to $z_s$.  Now $S$ is a minimal surface of general type with $p_g=0$, so we have $5K_S$ is very ample. So the  divisor $\sum_l pr_l^{-1}(5K_S)$ is ample on $S^g$, call it $D$. Then $D$ meets $F_s$, for all $z_s$. We have the zero cycle $z_s$ supported on $S^{g-k}\times C^k$, where $C$ is a smooth curve in the linear system of $5K_S$.

Indeed, Consider the divisor
$\sum_i \pi_i^{-1}(5K_{S})$ on the product $S^{g}$. This divisor is ample, so it intersects $F_s$, so we get that there exist points in $F_s$  contained in $C\times S^{g-1}$ where $C$ is in the linear system of $5K_{S}$. Then consider the elements of $F_s$ of the form $(c,s_1,\cdots,s_{g-1})$, where $c$ belongs to $C$. Considering the map from $S^{g-1}$ to $A_0(S)$ given by
$$(s_1,\cdots,s_{g-1})\mapsto \sum_j s_j+c-\sum_j i(s_j)-i(c)\;,$$
we see that this map factors through the product of the Prym fibration and the Jacobian of $D$ and the map from $S^{g-1}$ to $\bcP(\wt{\bcC}/\bcC)\times J(D)$ has positive dimensional fibers, since $i$ is large (here $i$ is the number of co-ordinates which contains points of $D.C_s$). So it means that, if we consider an element $(c,s_1,\cdots,s_{g-1})$ in $F_s$ and a curve through it, then it intersects the ample divisor given by $\sum_i \pi_i^{-1}(5K_{S})$, on $S^{g-1}$. Then we have some of $s_i$ is contained in $C$. So iterating this process we get that there exist elements of $F_s$ that are supported on $C^k\times S^{g-k}$, where $k$ is some natural number depending on $g$. Genus of the curve $C$ is $16$ or $31$ depending on $K_S^2=1,2$. We can choose the very ample line bundle $L$ to be such that $g$ is very large, so $k$ is larger than the genus of $C$. So we have $z_s$ is supported on $C^{16}\times S^{g-k},C^{31}\times S^{g-k}$ (respectively), hence we have that $\Gamma_*(U)=\Gamma_*(S^{i_0})$, where $i_0=16+g-k$ or $g-k+31$, which is strictly less than $g-1$, since the genus of $C$ is strictly less than $k$. That is we have proven that  for any $(s_1,\cdots,s_g)$ in $U$ in $S^g$, $z_s$ is rationally equivalent to a cycle on $S^i$. By using the fact that $\Gamma_*(U)$ is $\Gamma_*(S^g)$, we have proven that $\Gamma_*(S^g)=\Gamma_*(S^i)$.

Now we prove by induction that $\Gamma_*(S^i)=\Gamma_*(S^m)$ for all $m\geq g$. We can avoid the case when $g=i$, by adding sufficient large multiple of $L$ so that $g$ increases.
So suppose that $\Gamma_*(S^k)=\Gamma^*(S^i)$ for $k\geq g$, then we have to prove that $\Gamma_*(S^{k+1})=\Gamma_*(S^i)$. So any element in $\Gamma_*(S^{k+1})$ can be written as  $\Gamma_*(s_1+\cdots+s_i)+\Gamma_*(s)$. Now let $k-i=m$, then $i+1=k-m+1$. Since $k-m<k$, we have $k-m+1\leq k$, so $i+1\leq k$, so we have the cycle
$$\Gamma_*(s_1+\cdots+s_i)+\Gamma_*(s)$$
supported on $S^k$, hence on $S^i$. So we have that $\Gamma_*(S^i)=\Gamma_*(S^k)$ for all $k$ greater or equal than $g$. Now any element $z$ in $\CH_0(S)_{hom}$, can be written as a difference of two effective cycle $z^+,z^-$ of the same degree. Then we have
$$\Gamma_*(z)=\Gamma_*(z^+)-\Gamma_*(z_-)$$
and $\Gamma_(z_{\pm})$ belong to $\Gamma_*(S^i)$. So let $\Gamma'$ be the correspondence on $S^{2i}\times S$ defined as
$$\sum_{l\leq i}(pr_i,pr_S)^*\Gamma-\sum_{i\leq l\leq 2i}(pr_i,pr_S)^* \Gamma$$
where $\pr_i$ is the $i$-th projection from $S^i$ to $S$, and $\pr_S$ is from $S^i\times S$ to the last copy of $S$. Then we have
$$\im(\Gamma_*)=\Gamma'_*(S^{2i})\;.$$

\end{proof}

By the classification of Numerical Godeaux surfaces with an involution according to \cite{CCL}, we have that the quotient surface is either birational to an Enriques surface or it is birational to a rational surface. By applying the above theorem we get that:

\begin{theorem}
For all Numerical Godeaux surfaces with an involution such that the quotient is Enriques, the involution acts as identity on the group of algebraically trivial zero cycles modulo rational equivalence.
\end{theorem}

By the classification of Numerical Campedelli surfaces with an involution according to \cite{CLP}, when the bi-canonical map factors through the quotient map $S\to S/i$, then $S/i$ is birational to either an Enriques surface. Then by \cite{CLP}[Proposition 3.1] it follows that the fixed locus of the involution consists of a single irreducible component of genus $3$. So we have the following:

\begin{theorem}
For all Numerical Campedelli surfaces with an involution, such that the bicanonical map factors through the quotient map and the quotient is Enriques, the involution acts as identity on the group of algebraically trivial zero cycles modulo rational equivalence.
\end{theorem}

Consider a surface $S$ of general type with geometric genus equal to zero with $K^2=3$. Suppose such a surface has an involution $i$, such that the bicanonical map is not composed with the involution and the quotient is not rational, then by Theorem 5 of \cite{Ri}, we have the classification for the branch locus of the quotient map. Consider the case when this branch locus contains $5$ curves of self intersection $-1$ and an elliptic curve (in such a case the quotient is a surface not of general type with $p_g=0$). Then taking a very ample divisor on $S/i$ and arguing as in \ref{theorem2} we have:

\begin{theorem}
\label{theorem3}
The Bloch's conjecture holds for the surface $S$ with $K_S^2=3$ and having an involution such that the bicanonical map is not composed with an involution, with the components of the branch locus mentioned above.
\end{theorem}

Now consider a fake projective plane $S$, such that there is an automorphism of order three acting on it, say $g$ and $S/g$ is birational to a surface with geometric genus zero with $K^2=3$. Suppose further the desingularisation of the quotient surface has an involution not composed with the involution and having branch locus as in Theorem \ref{theorem3}. Then by Theorem 2.3 in \cite{CK} and Theorem 1.1 in \cite{K}, we have that the fixed locus of the map $S\to S/g$ is a set of three isolated points and the argument as in Theorem \ref{theorem2}, for a $6:1$ map (obtained as the composition of the maps $S\to S/g\to S/g/i$) instead of a $2:1$ shows that the automorphism acts as $\id$ on $S$. Also since the Bloch's conjecture is true on the desingularisation of the quotient $S/g/i$, we have $\id+g+g^2+g^3+g^4+g^5=\id$ on the group of algebraically trivial zero cycles. Therefore we have $5\id=0$ on $A_0(S)$, and hence by Roitman's theorem we have $A_0(S)=\{0\}$.

\end{document}